\documentclass[11pt]{article}
\usepackage{amssymb}
\usepackage{}
\usepackage{pifont}
\usepackage{amscd}
\usepackage{amsfonts}
\usepackage{amsmath, amsthm}
\usepackage{amssymb, amsxtra}

\usepackage{graphicx}

\usepackage{multicol}
\usepackage{epstopdf}
\usepackage{epsf}
\usepackage{float}
\usepackage{extarrows}
\usepackage[a4paper, text={.8\paperwidth,.83\paperheight}, ratio=1:1]{geometry}
\usepackage[format=hang,font=small,textfont=it]{caption}
\usepackage{authblk}

\usepackage[colorlinks=true, linkcolor=blue, citecolor=blue]{hyperref}
\usepackage[english]{babel}
\usepackage{latexsym}

\usepackage{tikz}
\usetikzlibrary{calc}
\usepackage[tight]{subfigure}

\usepackage{color}
\definecolor{greenbean}{RGB}{199,237,204}

\newtheorem{thm}{Theorem}[]

\newtheorem{lemma}[thm]{Lemma}
\newtheorem{prop}[thm]{Proposition}
\newtheorem{cor}[thm]{Corollary}
\newtheorem{eg}{Example}[]

\newtheorem{Def}[thm]{Definition}

\def \ta{\tau}

\def \ta1{\tau_1}

\begin{document}
\title{Deformations of Zappatic stable surfaces and their Galois covers
\footnotetext{\hspace{-1.8em} Email address: M. Amram: meiravt@sce.ac.il; C. Gong: cgong@suda.edu.cn; J.-L. Mo: mojiali0722@126.com. 
\\
2020 Mathematics Subject Classification. 05E14, 14J10, 14J25, 14N20.}}

\author[1]{Meirav Amram}
\author[2]{Cheng Gong}
\author[2]{Jia-Li Mo}


\vspace{2cm}

\affil[1]{\small{Shamoon College of Engineering, Ashdod, Israel}}
\affil[2]{\small{Department of Mathematics, Soochow University, Shizi RD 1, Suzhou 215006, Jiangsu, China}}

\date{}
\maketitle
\tableofcontents

\abstract This paper considers some algebraic surfaces that can deform to planar Zappatic stable surfaces with a unique singularity of type $E_n$. We prove that the Galois covers of these surfaces are all simply connected of general type, for $n \geq 4$, and we give a formula for Chern numbers of such Galois covers. As an application,  we prove that such surfaces do not exist for~$n>30$. Furthermore, Koll$\acute{a}$r improves the result to $n >9$ in Appendix \ref{A1}.


\section{Introduction}\label{outline}
 Unlike the moduli spaces of algebraic curves, the moduli spaces of surfaces have many components. Like the stable curves in moduli of curves, the stable surfaces play an important role in the moduli spaces of algebraic surfaces. 
In  \cite{C} and \cite{LR}, some papers pay attention to stable surfaces, especially in \cite{C}, where the author introduces two types of basic reduced surfaces: a string of surfaces and a cycle of surfaces. The above surfaces are unions of some planes.

Let $X$ be an algebraic surface in a complex projective space, and let $f: X \to \mathbb{CP}^2 $ be a generic projection.  Then we can define the Galois cover of $X$.
It is well known that the Galois cover of $X$ could reflect some topology properties of $X$, see \cite{Li}. Especially,   Galois covers could provide some new invariants (as mentioned in  \cite{Tei2}).
Studying the topological properties of Galois covers of surfaces becomes an important tool for studying the topological properties of algebraic surfaces. We can refer to \cite{Li}.

Zappa first studied some interesting surfaces in \cite{za2} and \cite{za1}, now called  Zappatic surfaces.
Later, more research focused on Zappatic surfaces, see \cite{C-C-F-M-2}, \cite{C-C-F-M-3},  \cite{C-C-F-M-1}, and \cite{C-C-F-M-4}. There are two types of the simplest stable Zappatic surfaces $X$ that have only one Zappatic singularity:  Zappatic surfaces of type $R_n$  and Zappatic surfaces of type $E_n$, see \cite{C-C-F-M-5} and \cite{C-C-F-M-1}. These two types of surfaces are actually a string of surfaces and a cycle of surfaces.

In \cite{AGM},  we studied Zappatic surfaces of type $R_n$.  In the language of \cite{C}, this is a string of planes that has only one common singularity.
We used the surfaces of the minimal degree to construct such deformations to Zappatic surfaces of type $R_n$, and proved that their Galois covers are always simply connected.

At the same time, some researchers also started to discuss Zappatic surfaces of type $E_n$ (in the language of \cite{C}, this is a cycle of planes that has only one common singularity). In \cite{Zappatic}, the authors proved that the Galois covers of the smooth algebraic surfaces deforming to  Zappatic surfaces of type $E_n$ are simply connected for small numbers. Specifically,  these surfaces can be deformed into Del Pezzo surfaces. In the same paper, a proposed question was whether the above theorem is true for $n>9$. 

In this paper, we present an answer to their question. The main theorem is elaborated in the following:
\begin{thm} \label{123}
If a smooth projective algebraic surface deforms to a Zappatic surface of type $E_{n}$, $n \geq 4$, then its Galois cover is simply connected of general type. 
\end{thm}

We also discuss the existence of the above deformation, and have the following corollary:
\begin{cor} 
If $n > 30$, the above deformation does not exist.
\end{cor} 

An interesting question centers on the existence of the above deformation between $10$ and $30$. From the computation in  Section \ref{4.2}, we can see that the existence of Zappatic surfaces of type $E_n$, for $n>11$,  shows that we can construct a surface of general type with a positive index. 
 J$\acute{a}$nos Koll$\acute{a}$r resolved the aforementioned existence problem, and his proof is given in Appendix \ref{A1}.

It is important to state that as our findings show, our goals were met and the mathematical community no longer needs to face the challenges that confronted us. The use of the results of the general case for such deformations can be spread to other studies about invariants of classification of algebraic surfaces. In addition, our work is dedicated to the significant development that consists of the original research, which was about deformations of specific cases and examples, while our contribution is for general cases and giving proofs about the existence of deformations.  

The paper is organized as follows:  Section \ref{method} gives the terminology and algorithm related to the work. Section \ref{E1010} shows the construction of the deformation for an algebraic surface of type $E_{10}$. Then we show the fundamental groups of the Galois covers of such surfaces. Section \ref{EnEn} considers the general Zappatic surfaces of type $E_{n}$, $n \geq 4$, each of which has a deformation with a unique  Zappatic singularity of type $E_n$. In Subsection \ref{4.1}, we determine the results of the fundamental groups of the Galois covers of these surfaces, while in Subsection \ref{4.2}, we calculate the topological invariants of these surfaces and study their existence. In Appendix \ref{A1},   J$\acute{a}$nos Koll$\acute{a}$r solved the existence problem of such surfaces.

To provide the fundamental group results and topological index formula for the Galois covers of general surfaces, assuming the existence of deformations, we obtain such surfaces for $n>11$; this demonstrates the construction of a surface of general type with a positive index. In fact, our proofs and all results hold for $n \geq 4$.

\paragraph{Acknowledgements:} We express our gratitude to  Prof. J$\acute{a}$nos Koll$\acute{a}$r for discussing the existence problem of stable Zappatic surfaces, and for his contribution.

This research was partly supported by NSFC~(12331001)~and the Natural Science Foundation of Jiangsu Province~(BK 20181427).

\section{Terminology and algorithm}\label{method}

To move towards the mathematical solution and to show that our hypothesis is correct, we must remember the terminology that includes the existence of Zappatic deformations of type $E_n$. We must also remember the construction of a fundamental group of the complement of the branch curve in $\mathbb{CP}^2$. We can then determine the fundamental groups of the Galois covers of such deformations, for each $n$. 

First, we give a precise definition of a Galois cover; we can refer to \cite{MoTe87}.
\begin{Def}
 Let $X\hookrightarrow\mathbb{CP}^n$ be an embedded algebraic surface. 
 Let $f : X \rightarrow  \mathbb{CP}^2 $ be a generic projection, $n = \mathrm{deg} f$, and let
$$X\times_{f} \cdots \times_{f} X = \{(x_1,\dots, x_n)| x_i\in X,~~ f(x_i) = f(x_j)~~\forall i,~~ j\};$$
$$\triangle = \{(x_1,\dots, x_n) \in  X\times_{f}\cdots\times_{f} X| x_i=x_j~~\mathrm{for~~some}~~i \neq j\}.$$
The Galois cover  $X_{Gal}$ of the generic projection w.r.t. the symmetric group $\mathcal{S}_n$
 is the closure of $\underbrace{X\times_{f}\cdots\times_{f}X}_{\mathrm{n~~times}} - \Delta$. 
\end{Def}

We introduce a good Zappatic surface with only one Zappatic singularity of type $E_n$, which appears in Figure \ref{En}. 

\begin{figure}[ht]
\begin{center}
\scalebox{0.34}{\includegraphics{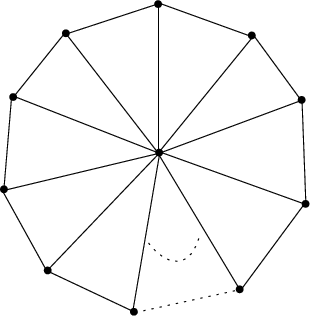}}
\end{center} \caption{A planar Zappatic deformation of type $E_n$}\label{En}
\end{figure}

\begin{Def}
    We denote the surfaces in this paper as $X_0^{E_n}$, called Zappatic surfaces of type $E_{n}$; these are unions of
$n$ planes in Figure \ref{En}, where a smooth projective algebraic surface $X^{E_n}$ deforms into   $X_0^{E_n}$.
\end{Def}

In \cite{Zappatic}, we determine the fundamental groups of the Galois covers of four surfaces with deformations $X_0^{E_n}$ for $n=6,7,8,9$. In Section \ref{E1010}, we study the case of $n=10$, then find a rule that works for each~$n$. 

Firstly, we look at Figure \ref{winfig5}, which has three parts. The first part (on the left side) is $X_0^{E_8}$,  with eight lines meeting at the Zappatic $E_8$ (which is denoted by $O_8$). This deformation will be our basic schematic figure, on which we add more and more lines to get the desired deformed surfaces. The middle picture is  $X_0^{E_9}$ and the right picture is  $X_0^{E_{10}}$. 
\begin{figure}[ht]
\begin{center}
\scalebox{0.4}{\includegraphics{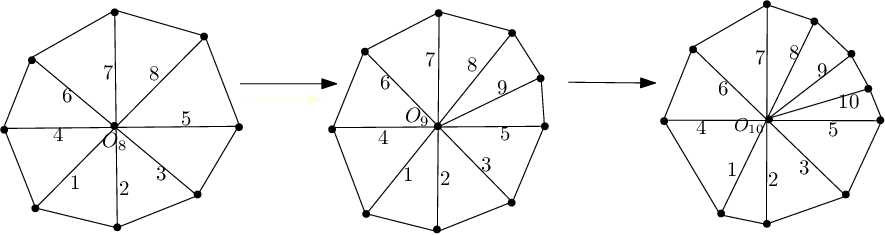}}
\end{center} \caption{Planar Zappatic deformations of type $E_n, n=8,9,10$}\label{winfig5}
\end{figure}

The lines in  $X_0^{E_{n}}$ are numbered from 1 to $n$, as shown in Figure \ref{XEn}. 
 The numbering of the lines in $X_0^{E_{n}}$ is from left to right in each row, going up from the bottom row to the upper row. Each addition of a line numbered as $i$, toward getting $X_0^{E_{n}}$, is done between the lines numbered as $5$ and $i-1$.  In Figure \ref{XEn} we depict $X_0^{E_{n}}$ with numbers of lines and notations of vertices. This figure is studied in detail, as shown in Section \ref{EnEn}.
\begin{figure}[ht]
\begin{center}
\scalebox{0.4}{\includegraphics{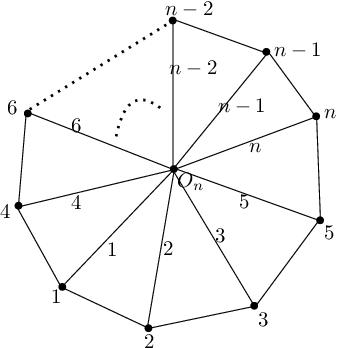}}
\end{center} \caption{A planar Zappatic deformation of type $E_n$}\label{XEn}
\end{figure}

Now we regenerate $X_0^{E_{n}}$  to $X^{E_{n}}$, which is an algebraic smooth surface in a complex projective space,  and we take $f: X^{E_{n}} \to \mathbb{CP}^2 $, a generic projection. We can get a branch curve $C^{E_{n}}$ of surface $X^{E_{n}}$. 
The branch curve $C^{E_{n}}$ has a degree $2n$ and is cuspidal. By the van-Kampen Theorem \cite{vk}, a presentation of the fundamental group  $G^{n}$ of the complement of $C^{E_{n}}$ in $\mathbb{CP}^2$ is accepted; this group has $2n$ generators  that are  $\gamma_1,\gamma_1',\dots,\gamma_{n},\gamma_{n}'$, and the following possible relations: 
\begin{enumerate}
\item $\gamma_{j} = \gamma_{j}'$ for a branch point in a conic.
\item $[\gamma_{i},\gamma_{j}]=[\gamma_{i},\gamma_{j}']=e$ for  nodes.
\item $\langle\gamma_{i},\gamma_{j}\rangle=\langle\gamma_{i},\gamma_{j}'\rangle=\langle\gamma_{i},\gamma_{j}^{-1}\gamma_{j}'\gamma_{j}\rangle=e$ for cusps.
\item In addition, the projective relation is $\gamma_{10}' \gamma_{10} \cdots \gamma_{1}'\gamma_1 = e$.
\end{enumerate}
Here, we denote $\gamma_{i}\gamma_{j}\gamma_{i}^{-1}\gamma_{j}^{-1}$ by $[\gamma_{i},\gamma_{j}]$ and $\gamma_{i}\gamma_{j}\gamma_{i}\gamma_{j}^{-1}\gamma_{i}^{-1}\gamma_{j}^{-1}$ by $\langle\gamma_{i},\gamma_{j}\rangle$.

We denote $G^{n}=\pi_1(\mathbb{CP}^2-C^{E_{n}})$ and define now   
$G^{n}_*= {G^{n}}/{<\gamma_{j}^2,\gamma_{j}'^2>}$. 
There is an evident surjection of this group onto the symmetric group $\mathcal{S}_{n}$, and the kernel of this map is the fundamental group of the Galois cover of the algebraic surface $X^{E_{n}}$, which deforms to  $X_0^{E_{n}}$ (see \cite{MoTe87}). 

Our goal is to prove the main theorem that the Galois covers of algebraic surfaces $X^{E_n}$  that deform to $X_0^{E_n}$ are all simply connected, and to check conditions on $n$. Relations and calculations are lengthy and, to some extent, also complex, as shown in the case of $X^{E_{10}}$.  Therefore, for the sake of convenience, we denote $\gamma_j$ by $j$ and $\gamma'_j$ by $j'$ in $G^n=\pi_1(\mathbb{CP}^2-C^{E_n}), ~n\geq4$. We define now   
$G^{n}_* = {G^n}/{<\gamma_{j}^2,\gamma_{j}'^2>}$ and denote the Galois cover of the Zappatic surface of type $E_{n}$ by $X_\text{Gal}^{E_n}$.

In addition, we use a proposition from \cite{MoTe87}, as follows:
\begin{prop}\label{prop1}
If $G^{n}_* \ \cong \ \mathcal{S}_n$, then the Galois cover of the surface is simply connected.
\end{prop}

\section{The study of the special case}\label{E1010}

{\bf The method of our proof is valid for $n \geq 4$.}

 In this section, we study the Galois cover of the generic fiber of the planar Zappatic deformation of type $E_{10}$, see Figure \ref{XE10}. This section provides useful rulings in the general cases that are presented in Section \ref{EnEn}.

First, to obtain a general formula, we may need to assume the existence of the planar Zappatic deformation of type $E_{10}$. Below are some introductions.

We return to the case of $X_0^{E_{10}}$, see Figure \ref{XE10} with numbers of lines and vertices. There is a generic projection from $X_0^{E_{10}}$ onto $\mathbb{CP}^2$. The branch curve in $\mathbb{CP}^2$ will be denoted by $C_0^{E_{10}}$. Each vertex $i$ corresponds to a line numbered as $i$, passing through it. The Zappatic singularity $O_{10}$ is in the middle of the picture.  
\begin{figure}[ht]
\begin{center}
\scalebox{0.4}{\includegraphics{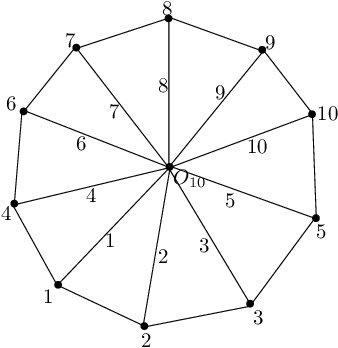}}
\end{center} \caption{A planar Zappatic deformation of type $E_{10}$}\label{XE10}
\end{figure}

We use a regeneration on the deformed Zappatic singularity; we explain this process now and give also \cite{MoTe87} as a reference. The initial step of the regeneration of a Zappatic singularity $O_{10}$ makes line $10$ a conic that is tangent to the neighboring lines numbered $5$ and $9$, see Figure \ref{E10T}. The remaining singularities in this picture are intersections of the conic with the other lines and a branch point of the conic. 
\begin{figure}[ht]
\begin{center}
\scalebox{0.4}{\includegraphics{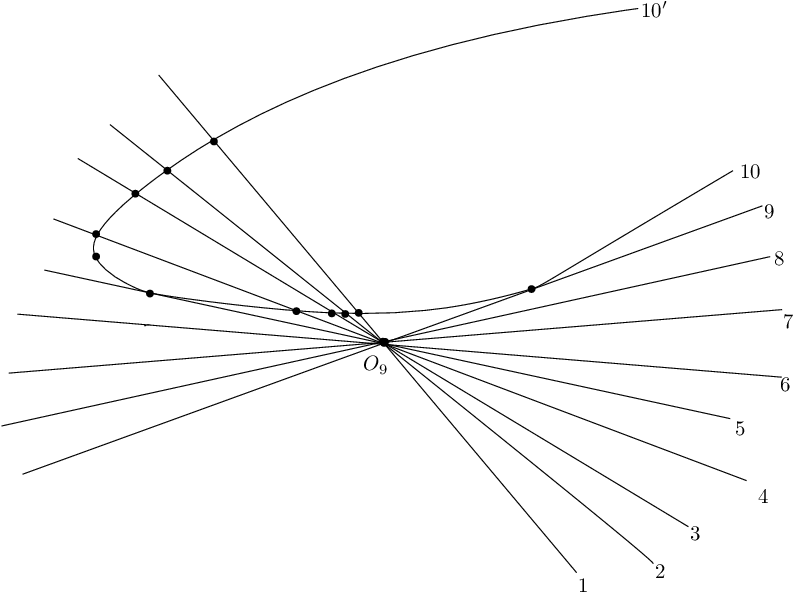}}
\end{center} \caption{Initial regeneration of $O_{10}$}\label{E10T}
\end{figure}

Then we regenerate the point that is an intersection of nine lines (which is actually $O_9$, appearing in $X_0^{E_9}$). In $O_9$, line $9$ regenerates to a conic that is tangent to lines numbered $5$ and $8$. Recursively, we can get the points that appear in lower degrees and that are already fully studied (i.e., $O_{k}, (k \leq 8)$).

In this section, we prove the following theorem:
\begin{thm}\label{thm10point}
The Galois cover of projective algebraic surfaces  $X^{E_{10}}$ is simply connected.
\end{thm}

\begin{proof}
Group $G^{10}$ is generated by ${\{j,j'\}}^{10}_{j=1}$ with the following relations: 
First, each vertex $i, \ \ i=1,2,\dots,10$ regenerates to a conic, giving rise to the relation $j=j'$ for $j=1,2,\dots,10$.
Vertex $O_{10}$ regenerates in the form shown in Figure \ref{E10T}. We can get all relations in group $G^{10}$. We show these relations in Appendix \ref{A2} because the relations are too long. To understand the lengthy presentation, the relations are divided into three parts, as explained in detail therein, and as discussed, briefly, below.

For the first part of the relations, we use $j=j'$ for $j=1,2,\dots,10$  to simplify relations (\ref{eq10.1})-(\ref{eq10.2}) from Appendix \ref{A2}, in $G^{10}_*$, and get:
\begin{equation}\label{eq10}
\begin{aligned}
& \langle 5, 10\rangle=e; ~~ \langle 9, 10\rangle=e,\\
& [i,10]=e, ~~i=1,2,3,4,6,7,8.
\end{aligned}
\end{equation}
We can find that (\ref{eq10}) establishes the simplest relations between generators $1,2,\dots,9$ and generator $10$.  

Next, for the second part of the relations, we use relations $j=j'$ for $j=1,2,\dots,10$  and (\ref{eq10}), to simplify relations (\ref{eq10.3})-(\ref{eq10.4}), then get the following relations in $G^{10}_*$:
\begin{equation}\label{eq9}
\begin{aligned}
& \langle 8, 9\rangle=e;~~ \langle 5, 10~9~10\rangle=e,\\
& [i,9]=e, ~~i=1,2,3,4,6,7.
\end{aligned}
\end{equation}
We pay note to (\ref{eq9}), which establishes the simplest relations between generators $1,2,\dots,8$, and $9$. Here, we suspect a relation of the form $[5,9]=e$ will be needed for the presentation (according to Figure \ref{XE10}, we are supposed to have such relation).

Finally, for the third part, using  $j=j'$ for $j=1,2,\dots,10$, together with (\ref{eq10}) and (\ref{eq9}), we simplify (\ref{eq10.5})-(\ref{eq10.6}) and obtain the following relations in group $G^{10}_*$:
\begin{equation}\label{eq8}
\begin{aligned}
& \langle 1, 2\rangle=\langle 1, 4\rangle=\langle 2, 3\rangle=\langle 3, 5\rangle=\langle 4, 6\rangle=\langle 6, 7\rangle=\langle 7, 8\rangle=e,\\
& \langle 5, 10~9~10~8~10~9~10\rangle=e,\\
& \langle 5, 10~9~10~8~10~9~10~7~10~9~10~8~10~9~10\rangle=e,\\
& 4~3~2~1~2~3~4 = 6~7~10~9~10~8~10~9~10~7~5~7~10~9~10~8~10~9~10~7~6, *\\
& [1,i]=e, ~~i=3,5,6,7,8; ~~[2,i]=e, ~i=6,7,8,\\
& [3,i]=e, ~~i=4,6,7,8; ~~[4,8]=[5,6]=[6,8]=e.
\end{aligned}
\end{equation}

Now, having (\ref{eq10}), (\ref{eq9}), and (\ref{eq8}), we try to produce $[5,8]=[5,9]=[2,4]=[2,5]=[4,5]=[4,7]=[5,7]=e$ as well. 
If we can prove that these missing relations actually exist in the presentation of $G^{10}_*$, then relations $\langle 5, 10~9~10\rangle=e$ in (\ref{eq9}), $\langle 5, 10~9~10~8~10~9~10\rangle=e$, and $\langle 5, 10~9~10~8~10~9~10~7~10~9~10~8~10~9~10\rangle=e$ in (\ref{eq8}), become simplified to $\langle 5, 10\rangle=e$, which can be found in (\ref{eq10}), so they will be redundant.

Relation $~4~3~2~1~2~3~4 = 6~7~10~9~10~8~10~9~10~7~5~7~10~9~10~8~10~9~10~7~6$ can be rewritten as  $1=2~4~3~6~5~7~10~9~10~8~10~9~10~7~5~6~3~4~2$ and we use it in $[1,9]=e$ to get $[5,9]=e$, as follows:
\begin{equation}\label{eq6.31}
\begin{aligned}
&e=[1,9]=[2~4~3~6~5~7~10~9~10~8~10~9~10~7~5~6~3~4~2,9]\\
&\xlongequal[\langle 8,9\rangle=e]{[i,9]=e, ~i=2,3,4,6,7}[10~9~10~5~10~9~10,9]
\xlongequal{\langle 9,10\rangle=e}[5,9].
\end{aligned}
\end{equation}
In the same way, relation $[1,8]=e$ can be written as $[5,8]=e$. 

Continuing with $*$ written as $7=10~9~10~8~10~9~10~5~6~3~4~1~2~1~4~3~6~5~10~9~10~8~10~9~10$, we get: 
\begin{equation}\label{eq8.1}
\begin{aligned}
& e=[1,7]=[1,10~9~10~8~10~9~10~5~6~3~4~1~2~1~4~3~6~5~10~9~10~8~10~9~10] \\
&\xlongequal {[1,3]=[1,5]=\cdots=[1,10]=e}[1,4~1~2~1~4]
\xlongequal {\langle 1,2\rangle=\langle 1,4\rangle=\langle 2,4\rangle=e }[2,4].
\end{aligned}
\end{equation}
In a similar way we can also derive relations $[2,5]=[4,5]=[4,7]=[5,7]=e$ from other existing relations: 
\begin{equation}\label{eq8.2}
\begin{aligned}
& [1,3]=e \Longrightarrow [2,5]=e; ~~[3,8]=e \Longrightarrow [4,5]=e,\\
& [2,6]=e \Longrightarrow [4,7]=e; ~~[3,7]=e \Longrightarrow [5,7]=e.
\end{aligned}
\end{equation}

Because $5=7~10~9~10~8~10~9~10~7~6~4~3~2~1~2~3~4~6~7~10~9~10~8~10~9~10~7$, we can eliminate it from the list of generators for $G^{10}_*$. It is obvious that $1,2,3,4,6,7,8,9$, and $10$ are the generators in $G^{10}_*$, making this the final presentation:
\begin{equation}\label{eq8.33}
\begin{aligned}
& \langle 1, 2\rangle=\langle 1, 4\rangle=\langle 2, 3\rangle=\langle 4, 6\rangle=\langle 6, 7\rangle=\langle 7, 8\rangle=e,\\
& [1,i]=e, ~~i=3,6,7,8,9,10; ~~[2,i]=e, ~i=4,6,7,8,9,10,\\
& [3,i]=e, ~~i=4,6,7,8,9,10; ~~[4,i]=e, ~i=7,8,9,10,\\
& [6,i]=e, ~~i=8,9,10; ~~[7,i]=e, ~~i=9,10; ~[8,10]=e.
\end{aligned}
\end{equation}

These are the same relations as those in the symmetric group $\mathcal{S}_{10}$; hence $G^{10}_*$ is isomorphic to $\mathcal{S}_{10}$. 
We note that the projective relation ${10}'{10}\cdots2'21'1=e$  is already trivial in $G^{10}$ because we have $j=j'$ for $j=1,\dots,10$. Proposition \ref{prop1} shows that the fundamental group of the Galois cover of the Zappatic surface of type $E_{10}$ is trivial.
\end{proof}

Now, we must organize some known facts to solve the general case. In \cite{Zappatic},  the Galois covers of the generic fibers of planar Zappatic deformations of type $E_n \ (n=6,7,8,9)$ are simply connected.

From detailed observation, we see that in \cite{Zappatic}, a similar process of Zappatic deformation occurs for both types $E_8$ and $E_9$. 
This paper presents a better-suited process for dealing with general cases. For a Zappatic deformation of type $E_{10}$,
we have similar rules as in the Zappatic deformation of type $E_9$, and we obtain some rules in the structure of the presentation of the fundamental group. 
We choose to consider first a Zappatic deformation of type $E_{8}$. The summary is as follows:

\begin{itemize}
  \item $E_8$: In group $G^{8}_*$, we found that the final simplified result has the relations $[2,4]=[2,5]=[4,5]=[4,7]=[5,7]=e$ (which should be part of our presentation) can be obtained by $[1,7]=[1,3]=[3,8]=[2,6]=[3,7]=e$ respectively.
  
  \item $E_9$: All relations appearing in group $G^{9}$ must be divided into two parts, then simplified in $G^{9}_*$.
  
  The relations in the first part are those involving generator $9$, as follows: $\langle 8, 9\rangle=e; ~~\langle 5, 9\rangle=e; ~~[i,9]=e, ~~i=1,2,3,4,6,7$.


The relations in the second part are $\langle 1, 2\rangle=\langle 1, 4\rangle=\langle 2, 3\rangle=\langle 3, 5\rangle=\langle 4, 6\rangle=\langle 6, 7\rangle=\langle 7, 8\rangle=e;
~~[1,i]=e, ~~i=3,5,6,7,8; ~~[2,i]=e, ~~i=6,7,8;
~~[3,i]=e,~~i=4,6,7,8; ~~[4,8]=[5,6]=[6,8]=e.$

Moreover, by relations  $[1,8]=e$ and $[1,7]=[1,3]=[3,8]=[2,6]=[3,7]=e$, we can obtain relations $[5,8]=e$, and $[2,4]=[2,5]=[4,5]=[4,7]=[5,7]=e$, respectively. 

Also important is that the elements in the relations of the second part will appear in relations of the Zappatic deformation of type $E_{8}$ with conjugations $8 \longrightarrow 9~8~9^{-1}$ and $8^{-1} \longrightarrow 9~8^{-1}9^{-1}$.

\item $E_{10}$: All relations in group $G^{10}$ must first be divided into three parts, then simplified in $G^{10}_*$.

 The relations in the first part are those involving generator $10$, as follows: $\langle 9, 10\rangle=e,  ~~\langle 5, 10\rangle=e$, and $[i,10]=e$ for $i=1,2,3,4,6,7,8$.


The relations in the second part are $\langle 8, 9\rangle=e, ~~\langle 5, 10~9~10\rangle=e$, and $[i,9]=e$ for $i=1,2,3,4,6,7$.

The relations in the third part are $\langle 1, 2\rangle=\langle 1, 4\rangle=\langle 2, 3\rangle=\langle 3, 5\rangle=\langle 4, 6\rangle=\langle 6, 7\rangle=\langle 7, 8\rangle=e$, 
~~$[1,i]=e$ for $i=3,5,6,7,8$,  ~~$[2,i]=e$ for $i=6,7,8$, 
~~$[3,i]=e$ for $i=4,6,7,8$,  ~~$[4,8]=[5,6]=[6,8]=e$, 
~~$\langle 5, 10~9~10~8~10~9~10\rangle=e$, 
~~$\langle 5, 10~9~10~8~10~9~10~7~10~9~10~8~10~9~10\rangle=e$, and 
$4~3~2~1~2~3~4 =6~7~10~9~10~8~10~9~10~7~5~7~10~9~10~8~10~9~10~7~6$.

Moreover, by relations $[1,8]=[1,9]=e$ and $[1,7]=[1,3]=[3,8]=[2,6]=[3,7]=e$  we can obtain relations $[5,8]=[5,9]=e$ and $[2,4]=[2,5]=[4,5]=[4,7]=[5,7]=e$, respectively.  

An important consideration is that the elements in the relations of the second part in $E_{10}$ will appear in relations of the first part of $E_{9}$, but with conjugations $9 \longrightarrow 10~9~10^{-1}$ and $9^{-1} \longrightarrow 10~9^{-1}10^{-1}$.
Moreover, the relations in the third part in $E_{10}$ will be the same as in the second part of $E_{8}$, but with conjugations $8 \longrightarrow (10~9~10^{-1})~8~(10~9^{-1}10^{-1})$ and $8^{-1} \longrightarrow (10~9~10^{-1})~8^{-1}(10~9^{-1}10^{-1})$.
\end{itemize}

\section{The study of the general case and conclusion}\label{EnEn}
The success of our general results depends on the proof of previous conclusions about the fundamental groups of Galois covers of good planar Zappatic deformations of type $E_{n}, n=6,7,8,9$, and we add the case of $n=10$ to confirm our results and strengthen the conclusions.  Because our goal is to find some rulings in those presentations, solving the question about $E_{10}$ in Section \ref{E1010} was necessary. This investigation should produce good results.

Subsection \ref{4.1} shows that the Galois covers of algebraic surfaces  $X^{E_{n}}$ ($n \geq 4$) are simply connected, and Subsection \ref{4.2} shows that they are of general type and have some conditions related to the existences of those surfaces. 
\subsection{Surfaces of the general case}\label{4.1}

\begin{thm}\label{thmnpoint}
The Galois covers of algebraic surfaces  $X^{E_{n}}$ ($n \geq 4$) are simply connected.
\end{thm}

\begin{proof}
The proof for the general case consists of Figure \ref{XEn}. Group $G^{n}$ has generators ${\{j,j'\}}^{n}_{j=1}$ with the following relations:
First, vertices $i, \ i=1,2,\dots,n$ give rise to relations $j=j'$ for $j=1,2,\dots,n$. 
Second, vertex $O_{n}$ deforms in the form shown in Figure \ref{EnT}. 
\begin{figure}[ht]
\begin{center}
\scalebox{0.4}{\includegraphics{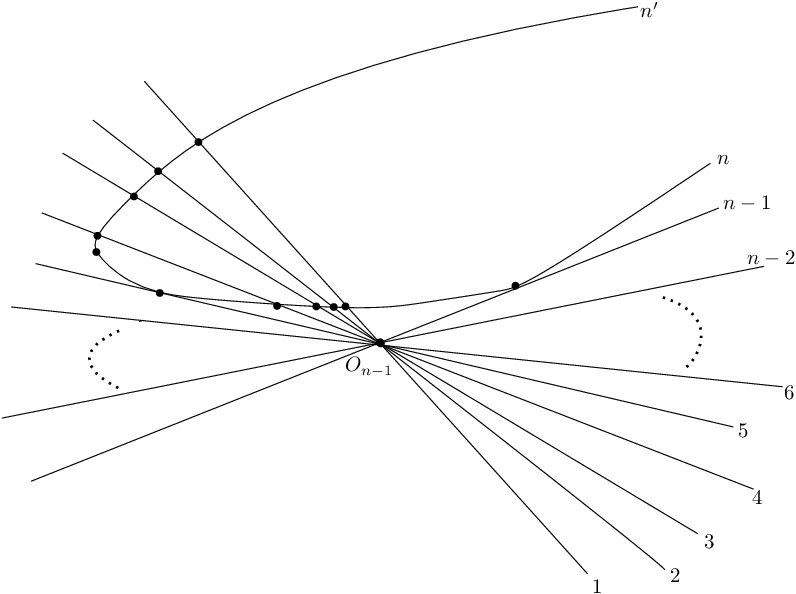}}
\end{center} \caption{A planar Zappatic deformation of type $E_{n}$}\label{EnT}
\end{figure}

We can get the relations from vertex $O_{n}$. All relations are divided into $n-7$ parts, to explain the rules appearing in those relations conveniently. Next, the first part of the relations from vertex $O_{n}$ is:
\begin{equation}\label{eqn1}
\langle j,~n\rangle=e,~~j=(n-1),~(n-1)',~(n-1)^{-1}(n-1)'(n-1),
\end{equation}
\begin{equation}
\begin{aligned}
&[(n-2)'(n-2)\cdots(j+1)'(j+1)~j~(j+1)^{-1}(j+1)'^{-1}\cdots\\
&(n-2)^{-1}(n-2)'^{-1}, n]=e,~~j=1,1'~2,2',~3,3',~4,4,
\end{aligned}
\end{equation}
\begin{equation}
\begin{aligned}
&[(n-1)'(n-1)\cdots(j+1)'(j+1)~j~(j+1)^{-1}(j+1)'^{-1}\cdots\\
&(n-1)^{-1}(n-1)'^{-1}, n^{-1}n'n]=e,~~j=1,1'~2,2',~3,3',~4,4,
\end{aligned}
\end{equation}
\begin{eqnarray}
&[j, n]=[j', n]=e,~~j=6, 7, 8, \dots, n-2
\end{eqnarray}
\begin{equation}
\begin{aligned}
&[j, n]=[j', n(j+1)^{-1}(j+1)'^{-1}(j+2)^{-1}(j+2)'^{-1}\cdots(n-1)^{-1}(n-1)'^{-1}n^{-1}n'n\\
&~(n-1)'(n-1)(j+2)'(j+2)(j+1)'(j+1)n^{-1}]=e,~~j=6, 7, 8, \dots, n-2,
\end{aligned}
\end{equation}
\begin{equation}
\begin{aligned}
 &\langle (n-2)'(n-2)(n-3)'(n-3)~j~(n-3)^{-1}(n-3)'^{-1}(n-2)^{-1}\\
 &(n-2)'^{-1},~n\rangle=e,~~j=(n-4),~(n-4)',~(n-4)^{-1}(n-4)'(n-4),
\end{aligned}
\end{equation}
\begin{equation}\label{eqn2}
\begin{aligned}
&n=(n-2)'(n-2)(n-3)'(n-3)(n-4)'(n-4)(n-3)^{-1}(n-3)'^{-1}(n-2)^{-1}\\
&(n-2)'^{-1}(n-1)^{-1}(n-1)'^{-1}n^{-1}n'n(n-1)'(n-1)(n-2)'(n-2)(n-3)'\\
&(n-3)(n-4)'(n-4)(n-3)^{-1}(n-3)'^{-1}(n-2)^{-1}(n-2)'^{-1}.
\end{aligned}
\end{equation}

We use  $j=j'$ to simplify relations (\ref{eqn1})-(\ref{eqn2}) in $G^{n}_*$, and get the following list that establishes the relations between generators $1,2,\dots,n-1$ and generator $n$:
\begin{equation}\label{eqn.11}
\begin{aligned}
& \langle 5, n\rangle=e; ~~\langle n-1, n\rangle=e,\\
& [j,n]=e, ~~j=1,2,3,4,6,7,\dots, n-2.
\end{aligned}
\end{equation}

Because the other parts of relations in $G^{n}$ are too long, the next parts of relations are given after simplifications in $G^{n}_*$, noting that the simplification is similar to the one above.

These are the relations in $G^{n}_*$ of the second part:
\begin{equation}\label{eqn.22}
\begin{aligned}
& \langle 5, n(n-1)n\rangle=e; ~~\langle n-2, n-1\rangle=e,\\
& [j,n-1]=e, ~~j=1,2,3,4,6,7,\dots, n-3.
\end{aligned}
\end{equation}

The relations in $G^{n}_*$ of the third part are:
\begin{equation}\label{eqn.33}
\begin{aligned}
& \langle 5, n(n-1)n(n-2)n(n-1)n\rangle=e;~~ \langle n-3, n-2\rangle=e,\\
& [j,n-2]=e, ~~j=1,2,3,4,6,7,\dots, n-4.
\end{aligned}
\end{equation}

$$\cdots\cdots\cdots\cdots\cdots$$

The relations in $G^{n}_*$ of the $(n-7)$-th part (from relations (\ref{eqn.4}) to (\ref{eqn.7})) are:    
\begin{eqnarray}\label{eqn.4}
&\langle 1, 2\rangle=\langle 1, 4\rangle=\langle 2, 3\rangle=\langle 3, 5\rangle=\langle 4, 6\rangle=\langle 6, 7\rangle=\langle 7, 8\rangle=e,\\
& [1,j]=e, ~~j=3,5,6,7,8; ~~[2,j]=e, ~j=6,7,8,\\
& [3,j]=e, ~~j=4,6,7,8; ~~[4,8]=[5,6]=[6,8]=e,
\end{eqnarray}

\begin{equation}\label{eqn.5}
\begin{aligned}
& \langle 5, ~n(n-1)n (n-2) n(n-1)n (n-3) n(n-1)n (n-2) n(n-1)n (n-4)\\
& n(n-1)n (n-2) n(n-1)n (n-3) n(n-1)n (n-2) n(n-1)n\cdots\cdots\\
&n(n-1)n (n-2) n(n-1)n (n-3) n(n-1)n (n-2) n(n-1)n (n-4) n(n-1)n (n-2)\\
& n(n-1)n (n-3) n(n-1)n (n-2) n(n-1)n~8~n(n-1)n (n-2) n(n-1)n (n-3)\\
& n(n-1)n (n-2) n(n-1)n (n-4) n(n-1)n (n-2) n(n-1)n (n-3) n(n-1)n\\
&(n-2) n(n-1)n\cdots\cdots n(n-1)n (n-2) n(n-1)n (n-3) n(n-1)n (n-2)\\
& n(n-1)n (n-4) n(n-1)n (n-2) n(n-1)n (n-3) n(n-1)n (n-2) n(n-1)n\rangle=e,
\end{aligned}
\end{equation}

\begin{equation}\label{eqn.6}
\begin{aligned}
& \langle 5, ~n(n-1)n (n-2) n(n-1)n (n-3) n(n-1)n (n-2) n(n-1)n (n-4)\\
& n(n-1)n (n-2) n(n-1)n (n-3) n(n-1)n (n-2) n(n-1)n\cdots\cdots\\
& n(n-1)n (n-2) n(n-1)n (n-3) n(n-1)n (n-2) n(n-1)n (n-4) n(n-1)n (n-2)\\
& n(n-1)n (n-3) n(n-1)n (n-2) n(n-1)n~8~n(n-1)n (n-2) n(n-1)n (n-3) \\
& n(n-1)n (n-2) n(n-1)n (n-4) n(n-1)n (n-2) n(n-1)n (n-3) n(n-1)n (n-2)\\
& n(n-1)n\cdots\cdots n(n-1)n (n-2) n(n-1)n (n-3) n(n-1)n (n-2) n(n-1)n\\
& (n-4) n(n-1)n (n-2) n(n-1)n (n-3) n(n-1)n (n-2) n(n-1)n~7~n(n-1)n (n-2)\\ & n(n-1)n (n-3) n(n-1)n (n-2) n(n-1)n (n-4) n(n-1)n (n-2) n(n-1)n (n-3)\\
& n(n-1)n (n-2) n(n-1)n\cdots\cdots n(n-1)n (n-2) n(n-1)n (n-3) n(n-1)n\\
&(n-2) n(n-1)n (n-4) n(n-1)n (n-2) n(n-1)n (n-3) n(n-1)n (n-2) n(n-1)n\\
&~8~n(n-1)n (n-2) n(n-1)n (n-3) n(n-1)n (n-2) n(n-1)n (n-4) n(n-1)n\\
&(n-2) n(n-1)n (n-3) n(n-1)n (n-2) n(n-1)n\cdots\cdots n(n-1)n (n-2) \\
& n(n-1)n (n-3) n(n-1)n (n-2) n(n-1)n (n-4) n(n-1)n (n-2) n(n-1)n\\
&(n-3) n(n-1)n (n-2) n(n-1)n\rangle=e,
\end{aligned}
\end{equation}

\begin{equation}\label{eqn.7}
\begin{aligned}
& 4~3~2~1~2~3~4 = 6~7~n(n-1)n (n-2) n(n-1)n (n-3) n(n-1)n (n-2) n(n-1)n \\
&(n-4) n(n-1)n (n-2) n(n-1)n (n-3) n(n-1)n (n-2) n(n-1)n\cdots\cdots\\
&n(n-1)n (n-2) n(n-1)n (n-3) n(n-1)n (n-2) n(n-1)n (n-4) n(n-1)n (n-2)\\
&n(n-1)n (n-3) n(n-1)n (n-2) n(n-1)n~8~n(n-1)n (n-2) n(n-1)n (n-3) \\
&n(n-1)n (n-2) n(n-1)n (n-4) n(n-1)n (n-2) n(n-1)n (n-3) n(n-1)n\\
&(n-2) n(n-1)n\cdots\cdots n(n-1)n (n-2) n(n-1)n (n-3) n(n-1)n (n-2)\\
&n(n-1)n (n-4) n(n-1)n (n-2) n(n-1)n (n-3) n(n-1)n (n-2) n(n-1)n\\
&~7~5~7~n(n-1)n (n-2) n(n-1)n (n-3) n(n-1)n (n-2) n(n-1)n (n-4) n(n-1)n\\
&(n-2) n(n-1)n (n-3) n(n-1)n (n-2) n(n-1)n\cdots\cdots n(n-1)n (n-2)\\
&n(n-1)n (n-3) n(n-1)n (n-2) n(n-1)n (n-4) n(n-1)n (n-2) n(n-1)n (n-3)\\
&n(n-1)n (n-2) n(n-1)n~8~n(n-1)n (n-2) n(n-1)n (n-3) n(n-1)n (n-2)\\
&n(n-1)n (n-4) n(n-1)n (n-2) n(n-1)n (n-3) n(n-1)n (n-2) n(n-1)n\\
&\cdots\cdots n(n-1)n (n-2) n(n-1)n (n-3) n(n-1)n (n-2) n(n-1)n\\
& (n-4) n(n-1)n (n-2) n(n-1)n (n-3) n(n-1)n (n-2) n(n-1)n~7~6.
\end{aligned}
\end{equation}
Moreover, we get relations $[5,8]=[5,9]=\cdots=[5,n-1]=e$ and $[2,4]=[2,5]=[4,5]=[4,7]=[5,7]=e$, from $[1,8]=[1,9]=\cdots=[1,n-1]=e$ and $[1,7]=[1,3]=[3,8]=[2,6]=[3,7]=e$, respectively.

In conclusion, we have the relations coming from all the above $(n-7)$ parts, after eliminating generator $5$: 
\begin{equation}\label{eq8.3}
\begin{aligned}
& \langle 1, 2\rangle=\langle 1, 4\rangle=\langle 2, 3\rangle=\langle 4, 6\rangle=\langle 6, 7\rangle=e,\\
& \langle 7, 8\rangle=\cdots=\langle n-2, n-1\rangle=\langle n-1, n\rangle=e,\\
& [1,j]=e, ~j=3,6,7,\dots,n; ~~[2,j]=e, ~j=4,6,7,\dots,n,\\
& [3,j]=e, ~j=4,6,7,\dots,n; ~~[4,j]=e, ~j=7,8,\dots,n,\\
& [6,j]=e, ~j=8,9,\dots,n; ~~[7,j]=e, ~j=9,10, \dots,n, \\
&\cdots\cdots\cdots\\
& [n-4,j]=e, ~j=n-2,n-1,n;~[n-3,j]=e, ~j=n-1,n; ~[n-2,n]=e.
\end{aligned}
\end{equation}

We note that the projective relation ${n}'{n}\cdots2'2~1'1=e$  is already trivial in $G^{n}$, because we have $j=j'$ for $j=1,2,\dots,n$. 
It is obvious that $\{1,2,3,4,6,7,8,\dots,n\}$ are generators in $G^{n}_*$. These are the same relations as those in the symmetric group $\mathcal{S}_{n}$; hence $G^{n}_*$ is isomorphic to $\mathcal{S}_{n}$.

By Proposition \ref{prop1}, it follows that the fundamental group of the Galois cover of $X^{E_n}$ is trivial (for $n \geq 4$).

\end{proof}

\subsection{Chern numbers}\label{4.2}

In this subsection, we use $c_1^2(X^{E_n}_{\text{Gal}})$ and $c_2(X^{E_n}_{\text{Gal}})$ to prove that the Galois covers of the surfaces $X^{E_n}$ are surfaces of general type. 
As a first step, we compute the Chern numbers. The formula was treated in \cite[Proposition 0.2]{MoTe87} (proof there is given by F. Catanese).

Let $X$ be an algebraic surface and $f: X\rightarrow \mathbb{CP}^2$ is the generic
projection. $C\subset \mathbb{CP}^2$ is the branch curve of $f$. We shall use the following notation:
 \begin{itemize}
   \item $m =$ degree of $C$; $d =$ number  of nodes in $C$; $\rho =$ number of cusps in $C$; $n =$ degree of $f$.
 \end{itemize}
We shall now get formulas for $ c_1^2(X^{E_n}_{\text{Gal}}) $ and $ c_2(X^{E_n}_{\text{Gal}}) $ in terms of $n, m, d, \rho$.

\begin{prop}(\cite{MoTe87}) \label{chern}
For $X, f, C, m, d, \rho$ as above, we have:
\begin{itemize}
  \item ~~$c_1^2(X_{\text{Gal}})=\frac{n!}{4} (m-6)^2$.
  \item ~~$c_2(X_{\text{Gal}})=n!(\frac{m^2}{2}-\frac{3m}{2}+3-\frac{3}{4}d-\frac{4}{3}\rho)$.
\end{itemize}
\end{prop}
Then, the signature $\tau(X_{\text{Gal}})=\frac{1}{3}(c_1^2(X_{\text{Gal}})-2c_2(X_{\text{Gal}}))$.

Note that  $n=degf$ and $m = 2n$. For surfaces $X^{E_n}$, we know that the branch curve is cuspidal. The numbers of cusps and nodes can be computed,  and method can be found in \cite{MoTe87}. Specifically, 
we get $\rho=6n-12,~d=2n^2-10n+12$. By Proposition \ref{chern}, we obtain

\begin{itemize}
  \item $c_1^2(X^{E_4}_{\text{Gal}})=4!\cdot1$;~~ $c_2(X^{E_4}_{\text{Gal}})=4!\cdot4$.
  \item $c_1^2(X^{E_5}_{\text{Gal}})=5!\cdot2^2$;~~ $c_2(X^{E_5}_{\text{Gal}})=5!\cdot5$.
  \item $c_1^2(X^{E_6}_{\text{Gal}})=6!\cdot3^2$;~~$c_2(X^{E_6}_{\text{Gal}})=6!\cdot7$.
  \item $\cdots\cdots\cdots$
  \item $c_1^2(X^{E_n}_{\text{Gal}})=n!\cdot{(n-3)^2}$;~~$c_2(X^{E_n}_{\text{Gal}})=n!\cdot(\frac{n^2}{2}-\frac{7}{2}n+10)$.
\end{itemize}
It is obvious that $c_1^2(X^{E_n}_{\text{Gal}})>0$ for $n \geq 4$, and this means that the Galois covers are surfaces of general type as explained in \cite[Theorem 1.1]{BHPV}. And $\tau(X_{\text{Gal}})=\frac{1}{3}(c_1^2(X_{\text{Gal}})-2c_2(X_{\text{Gal}}))=n!(\frac{n-11}{3})$.

Now, we continue our proof of Theorem \ref{123}.
For surfaces $X^{E_n}$, we get $\rho=6n-12,~d=2n^2-10n+12.$
It is easy to see that the genus~$g(C^{E_n})=\frac{(m-1)(m-2)}{2}-d-\rho=n+1$ in this case.
Now, we consider the dual curve~${C^{E_n}}^*$. It is easy to see that~$g(C^{E_n})=g({C^{E_n}}^*)$, and the degree of ${C^{E_n}}^*$ is $m^*$. Then,  $p$, $q$ are the number of nodes and cusps of ~${C^{E_n}}^*$, separately. 
By the Pl$\ddot{u}$ker formula:
\begin{equation} \label{eq:1}
\left\{ \begin{aligned}
         m^{*}=&m(m-1)-2d-3\rho \\
         n+1=&g(C)=\frac{(m^*-1)(m^*-2)}{2}-p-q \\
         m=&m^*(m^*-1)-2p-3q
                          \end{aligned} \right.
                          \end{equation}
We get $q=24, p=30-n$.
Because~$p\geq 0$, then $n>30 $, such surface $X^{E_n}$  does not exist.







Subsequently, in the following Appendix \ref{A1},  J$\acute{a}$nos Koll$\acute{a}$r demonstrates the following results that improve the initial  results:  these surfaces do not exist for $n>9$.

\section{Appendix}\label{A1}

The following alternate approaches were communicated to us by J$\acute{a}$nos Koll$\acute{a}$r.
\begin{thm}
Take a smooth projective algebraic surface that deforms to a Zappatic surface $X_0^{E_n}$ of type $E_n$.  
If $n > 9$, the above deformation does not exist.
\end{thm}

We give two proofs. 

{\bf The first proof:} 

The key point is that the dualizing sheaf $\omega_{X_0^{E_n}}$ is locally free and its inverse is ample. 
Given any flat family of surfaces  $Z \rightarrow D$ (for a disc $D$), whose central fiber is $X_0^{E_n}$, we know that $\omega_{Z/D}$ is flat. Since $\omega_{X_0^{E_n}}$ is locally free, so is $\omega_{Z/D}$; then the inverse of $\omega_{Z/D}$ is ample. So the smooth fibers are Del Pezzo surfaces. 

The details supporting the relationship of Del Pezzo surfaces and the smooth fibers follow: Several results that can be usually found only in rather general forms are required, and although they support simpler cases involving proofs that are much simpler, they are lacking written form.
Once the methodology is in place, the argument is somewhat better understood.

Fix $\mathbb{P}^n$, an origin $p = (0, \dots, 0, 1)$ and let $C_i$ be the coordinate axes through $p$. 
Let $S_i$ be the plane spanned by $C_i, C_{i+1}$. Then $\cup_i S_i$ is the $X_0^{E_n}$.

A simple but typical case is $X_0^{E_3} := (x_1x_2x_3 = 0)$. Thus $X_0^{E_3} $ is a cubic surface. For
a smooth cubic surface $S$, its canonical class $K_S$ is $-H$ (hyperplane class). The
claim is that essentially the same holds for $X_0^{E_n}$. The problem is that it is not 
normal, so the canonical class of such a variety is not defined in literature. 
However, one can use the sheaf version, the canonical (=dualizing) sheaf as
defined in \cite[Sec.III.7]{Hartshorne}, which is usually denoted by $\omega_{X_0^{E_n}}$. 

\begin{lemma} 
    $\omega_{X_0^{E_n}} = \mathcal{O}_{X_0^{E_n}} (-1)$, that is, the sheaf associated to minus the hyperplane class.
 \end{lemma}
\begin{proof}
One can write down this explicitly. On $S_i$, we take the local coordinates $x_i, x_{i+1}$. We then let the $2$-form be $(dx_i/x_i)\wedge(dx_{i+1}/x_{i+1})$. These match up along the axes $C_i$ and give a rational section of $\omega_{X_0^{E_n}}$ that has a simple pole along infinity: $X_0^{E_n}\cap(x_0 = 0)$. See a detailed explanation 
in \cite[Sec. 5.3]{Kollar}. 
\end{proof}

Let us now take a flat deformation of $X_0^{E_n}$. That projective morphism $ \pi : X \rightarrow \mathbb{D}$ (for $\mathbb{D}$=unit disc) is such that the central fiber is $X_0^{E_n}$. 

We remark that Calabri-Ciliberto-Flamini-Miranda \cite{C-C-F-M-5} say that ${X_0}^{E_n}$ is an assumed scheme theoretically.
If we also want deformations where embedded points are allowed in ${X_0} ^{E_n}$, perhaps these can be also classified.

Now we need to use the relative dualizing sheaf $\omega_{X/\mathbb{D}}$. A treatment that is not much more complicated than \cite[Section III.7]{Hartshorne} is given in \cite[Section 2.5]{Kollar2}. We also need to show that $X_0^{E_n}$ is Cohen-Macaulay, which is pretty mild for a surface.  We note that if a rational function is regular away from the vertex, then it is regular at the vertex. We view $X_0^{E_n}$ as the cone over $X_0^{E_n} \cap (x_0 = 0)$ and then use \cite[Section~3.1~]{Kollar}.  Then \cite[Theorem 2.68]{Kollar2} says that 
\begin{enumerate}
    \item $\omega_{X/\mathbb{D}}$ is flat over $\mathbb{D}$, and
\item  its restriction to the centeral fiber is isomorphic to ${X_0}^{{E_n}}$. 
\end{enumerate}
Thus, since $\omega_{E_n}$ is locally free, so is $\omega_{X/\mathbb{D}}$.
Since the inverse $\omega_{X_0^{E_n}}^{-1}$ is ample, so is the inverse $\omega_{X/\mathbb{D}}^{-1}$, on every fiber (perhaps after
shrinking $\mathbb{D}$), see  \cite[Theorem 1.2.17]{Lazar}.

Here we define Del Pezzo surfaces as follows.
\begin{Def}

Del Pezzo surfaces are defined to be surfaces $X$ of degree $d$ in $\mathbb{P}^d$ such that $\omega_X\cong \mathcal{O}_X(-1)$.
\end{Def}
Then, we have the following corollary:
\begin{cor}
1. The smooth fibers of $X \rightarrow \mathbb{D}$ are Del Pezzo surfaces of degree $n$.\\
2. Del Pezzo surfaces exist only for degrees $ \leq 9$ (cf. \cite[V.4.7.1]{Hartshorne}).
\end{cor}

This corollary completes our proof.

 {\bf The second proof}:

Because the surface $X_0^{E_n}$ has degree $n$ in $\mathbb{P}^n$, its smoothing also has degree $n$.
By \cite{D},
a surface of degree $n$ in $\mathbb{P}^n$ is either Del Pezzo
or it is elliptic ruled; see
\cite[Theorem 6.20]{C-C-F-M-5} for a modern proof.

An elliptic ruled surface has the holomorphic Euler characteristic $0$, so it
is not a flat deformation of $X_0^{E_n}$, which has the holomorphic Euler characteristic $1$.
For Del Pezzo  surfaces the degree is at most $9$, and we complete our proof.

We note that $\omega_{X_0^{R_n}}$ is not a line bundle for $n \geq 3$, so the above methods do not work for $X_0^{R_n}$. This type of singularity was treated already in \cite{AGM}.


\section{Appendix}\label{A2}
In this appendix, we have relations in group $G^{10}$, coming from vertex $O_{10}$ that appears in Figure \ref{XE10}. We divide these relations into three parts, as mentioned in Section \ref{E1010}.

The following expressions are the first part of relations in $G^{10}$:
\begin{equation}\label{eq10.1}
\langle j,~10\rangle=e,~~j=9,~9',~9^{-1}9'9,
\end{equation}
\begin{equation}
\begin{aligned}
&[8'8\cdots(j+1)'(j+1)~j~(j+1)^{-1}(j+1)'^{-1}\cdots8^{-1}8'^{-1},~10]=e, \\
&~~j=1,1'~,~2,2',~3,3',~4,4',
\end{aligned}
\end{equation}
\begin{equation}
\begin{aligned}
&[9'9\cdots(j+1)'(j+1)~j~(j+1)^{-1}(j+1)'^{-1}\cdots9^{-1}9'^{-1},\\
&~10^{-1}10'10]=e, ~~j=1,1'~,~2,2',~3,3',~4,4',
\end{aligned}
\end{equation}
\begin{eqnarray}
&[j,~10]=e,~~j=6,6'~,~7,7',~8,8',\\
&[j,~10~7^{-1}7'^{-1}8^{-1}8'^{-1}9^{-1}9'^{-1}
10^{-1}10'~10~9'98'87'7~10^{-1}]=e,~~j=6,6',\\
&[j,~10~8^{-1}8'^{-1}9^{-1}9'^{-1}10^{-1}10'~10~9'98'8~10^{-1}]=e,~~j=7,7',\\
&[j,~10~9^{-1}9'^{-1}10^{-1}10'~10~9'9~10^{-1}]=e, ~~j=8,8',\\
&\langle 8'87'7~j~7^{-1}7'^{-1}8^{-1}8'^{-1},10\rangle=e,~~j=6,~6',~6^{-1}6'6,
\end{eqnarray}
\begin{equation}\label{eq10.2}
10=8'87'76'67^{-1}7'^{-1}8^{-1}8'^{-1}9^{-1}9'^{-1}10^{-1}10'10~9'98'87'76'67^{-1}7'^{-1}8^{-1}8'^{-1}.
\end{equation}

The following expressions are the second part of relations in $G^{10}$:
\begin{equation}\label{eq10.3}
\langle j,~10~9~10^{-1}\rangle=e,~~j=8,~8',~8^{-1}8'8,\\
\end{equation}
\begin{equation}
\begin{aligned}
&[7'7\cdots(j+1)'(j+1)~j~(j+1)^{-1}(j+1)'^{-1}\cdots7^{-1}7'^{-1},\\
&10~9~10^{-1}]=e, ~~j=1,1'~,~2,2',~3,3',~4,4',
\end{aligned}
\end{equation}
\begin{eqnarray}
&[7'7~j~7^{-1}7'^{-1},~10~9~10^{-1}]=e,~6,6',\\
&[j,~10~9~10^{-1}]=e,~~j=6,6',
\end{eqnarray}
\begin{equation}
\begin{aligned}
&[8'8\cdots(j+1)'(j+1)~j~(j+1)^{-1}(j+1)'^{-1}\cdots8^{-1}8'^{-1},\\
&10~9^{-1}9'9~10^{-1}]=e, ~~j=1,1'~,~2,2',~3,3',~4,4',~6,6',~7,7',
\end{aligned}
\end{equation}
\begin{eqnarray}
&\langle 8'87'7~j~7^{-1}7'^{-1}8^{-1}8'^{-1},~10~9~10^{-1}\rangle=e, ~~j=6,~6',~6^{-1}6'6,
\end{eqnarray}
\begin{equation}\label{eq10.4}
\begin{aligned}
&10~9~10^{-1}=7'76'65^{-1}5'^{-1}6^{-1}6'^{-1}7^{-1}7'^{-1}8^{-1}8'^{-1}10~\\
&9^{-1}9'^{-1}9~10^{-1}8'87'76'65'56^{-1}6'^{-1}7^{-1}7'^{-1}.
\end{aligned}
\end{equation}

The following expressions are the third part, where we give all remaining relations in group $G^{10}$:
\begin{equation}\label{eq10.5}
\begin{aligned}
&\langle 1',~j\rangle=e,~~j=2,~2',~2^{-1}2'2,
\end{aligned}
\end{equation}
\begin{eqnarray}
&\langle j,~10~9~10^{-1}8~10~9^{-1}10^{-1}\rangle=e,~~j=7,~7',~7^{-1}7'7,\\
&[6'6\cdots2'21'~j~1'^{-1}2^{-1}2'^{-1}\cdots6^{-1}6'^{-1},~10~9~10^{-1}8~10~9^{-1}10^{-1}]=e,
~~j=2,~2',\\
&[2'21'2^{-1}2'^{-1}, ~10~9~10^{-1}8~10~9^{-1}10^{-1}~j~10~9~10^{-1}8^{-1}10~9^{-1}10^{-1}]=e,~~j=7,~7',\\
&[6'65'54'4~j~4^{-1}4'^{-1}5^{-1}5'^{-1}6^{-1}6'^{-1},~10~9~10^{-1}8~10~9^{-1}10^{-1}]=e,~~j=3,~3',\\
&[2'21'2^{-1}2'^{-1},~6]=[2'21'2^{-1}2'^{-1},6']=e,\\
&[6'65'5~j~5^{-1}5'^{-1}6^{-1}6'^{-1},~10~9~10^{-1}8~10~9^{-1}10^{-1}]=e,~~j=4,~4',\\
&[2'21'2^{-1}2'^{-1},5]=[2'21'2^{-1}2'^{-1},~5']=e,\\
&[6'65'52'21'2^{-1}2'^{-1}5^{-1}5'^{-1}6^{-1}6'^{-1},~10~9~10^{-1}8~10~9^{-1}10^{-1}]=e,\\
&[7'7\cdots2'21'j~1'^{-1}2^{-1}2'^{-1}\cdots7^{-1}7'^{-1}, ~10~9~10^{-1}8^{-1}8'8~10~9^{-1}10^{-1}]=e,~~j=2,~2',\\
&[1,~10~9~10^{-1}8~10~9^{-1}10^{-1}j~10~9~10^{-1}8^{-1}10~9^{-1}10^{-1}]=e,~~j=7,~7',\\
&[7'7\cdots4'434^{-1}4'^{-1}\cdots7^{-1}7'^{-1}, ~10~9~10^{-1}8^{-1}8'8~10~9^{-1}10^{-1}]=e,~~j=3,~3',\\
&[1,6]=[1,6']=e,\\
&[7'7\cdots5'5~j~5^{-1}5'^{-1}\cdots7^{-1}7'^{-1}, ~10~9~10^{-1}8^{-1}8'8~10~9^{-1}10^{-1}]=e,~~j=4,~4',\\
&[1,~5]=[1,~5']=e,\\
&[7'76'65'52'21'2^{-1}2'^{-1}5^{-1}5'^{-1}6^{-1}6'^{-1}7^{-1}7'^{-1},~10~9~10^{-1}8^{-1}8'8~10~9^{-1}10^{-1}]=e,\\
&[6'65'515^{-1}5'^{-1}6^{-1}6'^{-1},~10~9~10^{-1}8~10~9^{-1}10^{-1}]=e,\\
&[7'76'65'515^{-1}5'^{-1}6^{-1}6'^{-1}7^{-1}7'^{-1},~10~9~10^{-1}8^{-1}8'8~10~9^{-1}10^{-1}]=e,\\
&\langle 2'21'2^{-1}2'^{-1},~j\rangle=e,~~j=4,~4',~4^{-1}4'4,\\
&\langle6'6~j~6^{-1}6'^{-1},~10~9~10^{-1}8~10~9^{-1}10^{-1}\rangle=e,~~j=5,~5',~5^{-1}5'5,\\
&1=4'42'21'2^{-1}2'^{-1}4^{-1}4'^{-1},\\
&[1,~2'21'^{-1}2^{-1}2'^{-1}j~2'21'2^{-1}2'^{-1}]=e, ~~j=3,~3',\\
&[2'21'2^{-1}2'^{-1},~j]=e,  ~~j=3,~3',
\end{eqnarray}
\begin{equation}
\begin{aligned}
&6'65^{-1}5'^{-1}6^{-1}6'^{-1}7^{-1}7'^{-1}10~9~10^{-1}8^{-1}8'8~10~9^{-1}10^{-1}7'76'65'56^{-1}6'^{-1}=\\
&10~9~10^{-1}8~10~9^{-1}10^{-1},
\end{aligned}
\end{equation}
\begin{eqnarray}
&[6,~10~9~10^{-1}8~10~9^{-1}10^{-1}]=[6',10~9~10^{-1}8~10~9^{-1}10^{-1}]=e,
\end{eqnarray}
\begin{equation}
\begin{aligned}
&[6,~10~9~10^{-1}8~10~9^{-1}10^{-1}7^{-1}7'^{-1}10~9~10^{-1}\\
&8^{-1}8'8~10~9^{-1}10^{-1}7'7~10~9~10^{-1}8^{-1}10~9^{-1}10^{-1}]=e,~~j=6,~6',
\end{aligned}
\end{equation}
\begin{eqnarray}
&[3',~4]=[3',~4']=e,\\
&\langle j,~2'21'^{-1}2^{-1}2'^{-1}32'21'2^{-1}2'^{-1}\rangle=e,~~j=2,~2',~2^{-1}2'2,\\
&\langle4'43'4^{-1}4'^{-1},~j\rangle=e,~~j=5,~5',~5^{-1}5'5,\\
&[3',~4^{-1}4'^{-1}5^{-1}5'^{-1}4'44^{-1}5'54'4]=[3',4^{-1}4'^{-1}5^{-1}5'^{-1}4'44'4^{-1}4'^{-1}5'54'4]=e,\\
&3=2'21'2^{-1}2'^{-1}1'^{-1}2^{-1}2'^{-1}3^{-1}5'54'43'4^{-1}4'^{-1}5^{-1}5'^{-1}32'21'2'21'^{-1}2^{-1}2'^{-1},\\
&[3,~2'21'2^{-1}2'^{-1}1'^{-1}2^{-1}2'^{-1}\cdots5^{-1}5'^{-1}~j~5'5\cdots2'21'2'21'^{-1}2^{-1}2'^{-1}]=e, ~~j=6,~6',\\
&[3',~4^{-1}4'^{-1}5^{-1}5'^{-1}~j~5'54'4]=e, ~~j=6,~6',
\end{eqnarray}
\begin{equation}
\begin{aligned}
&[3,~2'21'2^{-1}2'^{-1}1'^{-1}2^{-1}2'^{-1}\cdots6^{-1}6'^{-1}10~9~10^{-1}8~10~9^{-1}10^{-1}j\\
&10~9~10^{-1}8^{-1}10~9^{-1}10^{-1}6'6\cdots2'21'2'21'^{-1}2^{-1}2'^{-1}]=e, ~~j=7,~7',\\
\end{aligned}
\end{equation}
\begin{equation}
\begin{aligned}
&[3',~4^{-1}4'^{-1}\cdots6^{-1}6'^{-1}10~9~10^{-1}8~10~9^{-1}10^{-1}j\\
&~10~9~10^{-1}8^{-1}10~9^{-1}10^{-1}6'6\cdots4'4]=e,~j=7,~7',
\end{aligned}
\end{equation}
\begin{eqnarray}
&[5,~6]=[5',~6]=e,\\
&\langle 6',~10~9~10^{-1}8~10~9^{-1}10^{-1}j~10~9~10^{-1}8^{-1}10~9^{-1}10^{-1}\rangle=e, ~~j=7,~7',7^{-1}7'7\\
&\langle j,~6\rangle=e,~~j=4,~4',~4^{-1}4'4,\\
&[j,~4^{-1}4'^{-1}5^{-1}5'^{-1}65'54'4]=e, ~~j=2,~2',
\end{eqnarray}
\begin{equation}
\begin{aligned}
&6=4^{-1}4'^{-1}5^{-1}5'^{-1}6^{-1}10~9~10^{-1}8~10~9^{-1}10^{-1}7'7~10~9~10^{-1}8^{-1}10~9^{-1}10^{-1}6'\\
&10~9~10^{-1}8~10~9^{-1}10^{-1}7^{-1}7'^{-1}10~9~10^{-1}8^{-1}10~9^{-1}10^{-1}65'54'4,
\end{aligned}
\end{equation}
\begin{eqnarray}
&[j,~2'21'^{-1}2^{-1}2'^{-1}3^{-1}632'21'2^{-1}2'^{-1}]=e,~~j=2,~2',
\end{eqnarray}
\begin{equation}
\begin{aligned}
&[j,~21'^{-1}2^{-1}2'^{-1}3^{-1}10~9
10^{-1}8~10~9^{-1}10^{-1}7'7~10~9~10^{-1}8^{-1}10~9^{-1}10^{-1}6'
10~9~10^{-1}8\\
&~10~9^{-1}10^{-1}7^{-1}7'^{-1}10~9~10^{-1}8^{-1}10~9^{-1}10^{-1}32'21'2^{-1}]=e,~~j=2,~2',
\end{aligned}
\end{equation}
\begin{eqnarray}
&\langle 2',~21'^{-1}2^{-1}2'^{-1}3^{-1}~j~32'21'2^{-1}\rangle=e,~~j=4,~4',~4^{-1}4'4, \\
&\langle 2,~2'21'^{-1}2^{-1}2'^{-1}3^{-1}~j~32'21'2^{-1}2'^{-1}\rangle=e,~~j=4,~4',~4^{-1}4'4,\\
&\langle j,~6^{-1}10~9~10^{-1}8~10~9^{-1}10^{-1}7~10~9~10^{-1}8^{-1}10~9^{-1}10^{-1}6\rangle=e,~~j=5,~5',~5^{-1}5'5,\\
&\langle j,~6^{-1}10~9~10^{-1}8~10~9^{-1}10^{-1}7^{-1}7'7~10~9~10^{-1}8^{-1}10~9^{-1}10^{-1}6\rangle=e,
~~j=5,~5',~5^{-1}5'5,
\end{eqnarray}
\begin{equation}
\begin{aligned}
&4=32'21'2^{-1}2'^{-1}21'^{-1}2^{-1}2'^{-1}3^{-1}4^{-1}4'^{-1}6^{-1}10~9~10^{-1}8\\
&10~9^{-1}10^{-1}7~10~9~10^{-1}8^{-1}10~9^{-1}10^{-1}65'6^{-1}10~9~10^{-1}8~10\\
&~9^{-1}10^{-1}7^{-1}10~9~10^{-1}8^{-1}10~9^{-1}10^{-1}64'432'21'2^{-1}2'2\\
&1'^{-1}2^{-1}2'^{-1}3^{-1},
\end{aligned}
\end{equation}
\begin{equation}
\begin{aligned}
&4'=32'21'2^{-1}2'^{-1}2^{-1}2'21'^{-1}2^{-1}2'^{-1}3^{-1}4^{-1}4'^{-1}6^{-1}10\\
&~9~10^{-1}8~10~9^{-1}10^{-1}7^{-1}7'7~10~9~10^{-1}8^{-1}10~9^{-1}10^{-1}656^{-1}\\
&10~9~10^{-1}8~10~9^{-1}10^{-1}7^{-1}10~9~10^{-1}8^{-1}10~9^{-1}10^{-1}64'43\\
&2'21'2^{-1}2'^{-1}22'21'^{-1}2^{-1}2'^{-1}3^{-1},
\end{aligned}
\end{equation}
\begin{equation}
\begin{aligned}
&4'=32'21'2^{-1}2'^{-1}21'^{-1}2^{-1}2'^{-1}3^{-1}4^{-1}4'^{-1}6^{-1}10~\\
&9~10^{-1}8~10~9^{-1}10^{-1}7~10~9~10^{-1}8^{-1}10~9^{-1}10^{-1}656^{-1}10\\
&~9~10^{-1}8~10~9^{-1}10^{-1}7^{-1}10~9~10^{-1}8^{-1}10~9^{-1}\\
& 10^{-1}64'432'21'2^{-1}2'21'^{-1}2^{-1}2'^{-1}3^{-1},
\end{aligned}
\end{equation}
\begin{equation}
\begin{aligned}
&4=32'21'2^{-1}2'^{-1}2^{-1}2'21'^{-1}2^{-1}2'^{-1}3^{-1}4^{-1}4'^{-1}6^{-1}10~\\
&9~10^{-1}8~10~9^{-1}10^{-1}7^{-1}7'7~10~9~10^{-1}8^{-1}10~9^{-1}10^{-1}65'6^{-1}\\
&10~9~10^{-1}8~10~9^{-1}10^{-1}7^{-1}10~9~10^{-1}8^{-1}10~9^{-1}10^{-1}64'432'21'\\
&2^{-1}2'^{-1}22'21'^{-1}2^{-1}2'^{-1}3^{-1},
\end{aligned}
\end{equation}
\begin{equation}
\begin{aligned}
&[2',~21'^{-1}2^{-1}2'^{-1}3^{-1}4^{-1}
4'^{-1}6^{-1}10~9~10^{-1}8~10~9^{-1}10^{-1}7\\
&10~9~10^{-1}8^{-1}10~9^{-1}10^{-1}64'432'21'2^{-1}]=e,
\end{aligned}
\end{equation}
\begin{equation}
\begin{aligned}
&[2,~2'21'^{-1}2^{-1}2'^{-1}3^{-1}4^{-1}
4'^{-1}6^{-1}10~9~10^{-1}8~10~9^{-1}10^{-1}7^{-1}7'7\\
&10~9~10^{-1}8^{-1}10~9^{-1}10^{-1}64'432'21'2^{-1}2'^{-1}]=e,
\end{aligned}
\end{equation}
\begin{equation}
\begin{aligned}
&[2,~2'21'^{-1}2^{-1}2'^{-1}3^{-1}6^{-1}10~9
10^{-1}8~10~9^{-1}10^{-1}7\\
&10~9~10^{-1}8^{-1}10~9^{-1}10^{-1}632'21'2^{-1}2'^{-1}]=e,
\end{aligned}
\end{equation}
\begin{equation}\label{eq10.6}
\begin{aligned}
&[2',~22'21'^{-1}2^{-1}2'^{-1}3^{-1}6^{-1}10~9~10^{-1}8~10
9^{-1}10^{-1}7^{-1}7'7\\
&10~9~10^{-1}8^{-1}10~9^{-1}10^{-1}632'21'2^{-1}2'^{-1}2^{-1}]=e.
\end{aligned}
\end{equation}
\end{document}